\documentclass[11pt,twoside]{article}
\usepackage{cite}
\usepackage{amsmath, amsthm, amssymb}
\newtheorem{definition}{Definition}
\newtheorem{theorem}{Theorem}
\newtheorem{lemma}{Lemma}
\newtheorem{remark}{Remark}

\newtheorem{proposition}{Proposition}
\makeatletter \oddsidemargin0.45in \evensidemargin \oddsidemargin
\begin{document}
\thispagestyle{empty} \setcounter{page}{1}



\begin{center}
{\Large\bf   Sturm Liouville Equations in the frame of  fractional operators  with  Mittag-Leffler kernels and their discrete versions}

\vskip.20in
  Raziye Mert$^{a}$, Thabet Abdeljawad$^{b}$,  Allan Peterson$^{c}$\\[2mm]
{\footnotesize $^{a}$Mechatronic Engineering Department, University of Turkish Aeronautical Association\\ 06790, Ankara, Turkey\\
$^{b}$Department of Mathematics and General Sciences,
Prince Sultan
University\\ P. O. Box 66833,  11586 Riyadh, Saudi Arabia\\
$^{c}$Department of Mathematics, University of Nebraska-Lincoln, Lincoln, NE, 68588-0130,U.S.A. }

\end{center}

\vskip.2in

{\footnotesize \noindent {\bf Abstract.}

Very recently, some authors have studied new types of fractional derivatives whose kernels are nonsingular. In this article, we study  Sturm-Liouville Equations ($SLEs$) in the frame of fractional operators with Mittag-Leffler kernels. We formulate some Fractional  Sturm-Liouville Problems ($FSLPs$) with the diffferential part containing the left and right sided derivatives. We investigate the self-adjointness, eigenvalue and eigenfunction properties of the corresponding  Fractional Sturm-Liouville Operators ($FSLOs$) by using
fractional integration by parts formulas. The nabla discrete version of our results are also established.

}


\vskip.1in

\section{Introduction and Preliminaries} \label{s:1}

Fractional calculus has been studied in the last two decades or so. It has been used effectively in the modelling of many problems in various fields of science and engineering. It has reflected successfully the description of the properties of non-local complex systems \cite{Samko, Dumitru}. On the other hand, the discrete fractional calculus was of interest among several mathematicians  \cite{Nabla, CP} and has been developing rapidly. For the sake of finding more fractional operators with different kernels, recently some authors have introduced and studied new non-local derivatives with non-singular kernels and have applied them successfully to some real world problems \cite{Abdon, FCaputo, Losada,Thabet1, ThabetROMP,CF}. The extension to higher order  fractional operators and their Lyapunov type inequalities have been investigated in \cite{JIA Lyapunov, ADE Lyapunov}.
The proposed kernels are non-singular such as those with Mittag-Leffler kernels. The approach in defining such operators is different from the one of classical fractional operators which is through an iterative process of either the usual integration or differentiation. The idea behind this is to define the fractional derivatives first by imposing a non-singular kernel depending on the degree $\alpha$ so that as $\alpha\rightarrow 1$  the usual derivative is obtained and then  by applying a Laplace transform method to find their corresponding fractional integrals. What makes those fractional derivatives with Mittag-Leffler  kernels more interesting is that their corresponding fractional integrals contain Riemann-Liouville  fractional integrals as a part of their structure.  The advantage of such operators is that they enable numerical analysts to develop more efficient algorithms in solving fractional dynamical systems by concentrating only on the coefficients of the differential equations rather than worrying about the singularity of the kernels in the case of the classical fractional operators \cite{CF1,CF2,CF3}.
Later, the discrete counterparts of these  fractional operators were introduced, studied, and their monotonicity properties were analyzed \cite{TQ CAM,Thabet2,ADE Monot,Chaos Monot, DDNS Lyapunov,Th EPJ}.

The $SLEs$, which were investigated a long time ago, have many applications in various areas of science, engineering, and mathematics \cite{Zettl,Boyce}. However, its formulation in the frame of classical fractional calculus has started very recently \cite{Rivero, Klimek}. The classical Sturm-Liouville problem ($SLP$) for a linear  differential equation of second order is a boundary value problem ($BVP$) of the form:

 \begin{eqnarray*}\label{SLE}
  &&-\frac{d}{dt}\left(p(t) \frac{dx}{dt}\right)+q(t)x(t)=\lambda r(t) x(t),~~t \in [a,b],\\
   &&c_1 x(a)+c_2 x^\prime(a)=0,\\
  &&d_1 x(b)+d_2 x^\prime(b)=0,
\end{eqnarray*}
where $p,p^\prime, q,r$ are continuous functions on the interval $[a,b]$ such as $p(t)>0$, $r(t)>0$ on $[a,b].$ The differential equation can be written in the form $$L(x)=\lambda r(t)x,$$
where $L(x)=-[p(t) x^{\prime}]^{\prime}+q(t)x.$ A $\lambda$ for which the above $BVP$ has a nontrivial solution is called an eigenvalue, and the corresponding solution, an eigenfunction.

Motivated by what we have mentioned above, we introduce and analyze fractional $SLEs$ in the  frame of fractional operators with Mittag-Leffler kernels and their discrete counterparts. The corresponding fractional operator $L$ is introduced so that it contains left and right sided fractional operators with Mittag-Leffler kernels which makes possible to apply the suitable integration by parts formulas presented in \cite{Thabet1, Thabet2}.


This article is  organized as follows: In the rest of this section, we recall some basic concepts concerning the classical fractional calculus, classical nabla discrete fractional calculus, fractional operators with Mittag-Leffler kernels, and discrete fractional operators with discrete Mittag-Leffler kernels. In section 2, we state the main results which is divided into two parts. The first part is devoted to $SLEs$ in the frame of fractional operators with  Mittag-Leffler kernels  and the second part to $SLEs$ in the frame of nabla discrete fractional operators with  discrete Mittag-Leffler kernels. Finally, in section 3, we present an open problem for the higher order discrete fractional $SLE$ of order $\alpha \in (1,\frac{3}{2})$.

Below, we first recall some basic concepts from classical fractional calculus.

\begin{definition} (\cite{Podlubny}) The Mittag-Leffler function of one parameter is defined by

\begin{equation*}\label{classical ML}
  E_\alpha(z)= \sum_{k=0}^\infty \frac{z^k}{\Gamma(\alpha k+1)},\,\,\alpha\in \mathbb{C}, Re(\alpha)>0,
\end{equation*}
and the one with two parameters $\alpha$ and $\beta$ by

\begin{equation*}\label{cl ML2}
 E_{\alpha,\beta}(z)= \sum_{k=0}^\infty \frac{z^k}{\Gamma(\alpha k+\beta)},\,\alpha,\beta \in \mathbb{C}, Re(\alpha)>0, Re(\beta)>0,
\end{equation*}
where $E_{\alpha,1}(z)=E_{\alpha}(z)$.
\end{definition}

\begin{definition} (\cite{Kilbas})
The generalized Mittag-Leffler function of three parameters is defined by
\begin{equation*}\label{ML3}
  E^\rho_{\alpha,\beta} ( z)=\sum_{k=0}^\infty
\frac{(\rho)_kz^k} { k! \Gamma(\alpha k+\beta)},\,\alpha,\beta,\rho\in \mathbb{C}, Re(\alpha)>0, Re(\beta)>0, Re(\rho)>0,
\end{equation*}
where $(\rho)_k=\frac{\Gamma(\rho+k)}{\Gamma(\rho)}.$
\end{definition}
\noindent Notice that $(1)_k=k!$ so $ E^1_{\alpha,\beta} (z)=E_{\alpha,\beta}(z)$.

\begin{itemize}
  \item The left  fractional integral of order $\alpha >0$ starting at  $a$ has the following form
   $$(~_{a}I^\alpha f)(t)=\frac{1}{\Gamma(\alpha)}\int_a^t (t-s)^{\alpha-1}f(s)ds.$$
  \item The right fractional integral of order $\alpha >0$ ending at $b$ is defined by
   $$(I_b^\alpha f)(t)=\frac{1}{\Gamma(\alpha)}\int_t^b (s-t)^{\alpha-1}f(s)ds.$$
  \item The left Riemann-Liouville fractional derivative of order $0<\alpha <1$ starting at $a$ is given by
  $$(~_{a}D^\alpha f)(t)=\frac{d}{dt}(~_{a}I^{1-\alpha} f)(t).$$
  \item The right Riemann-Liouville fractional derivative of order $0<\alpha <1$ ending  at $b$ has the form
  $$(D_b^\alpha f)(t)=\frac{-d}{dt}(I_b^{1-\alpha} f)(t).$$
\end{itemize}

\begin{definition} (\cite{Abdon})
Let $f \in H^1(a,b),~ a<b,~\alpha \in [0,1]$. Then the left Caputo fractional derivative with Mittag-Leffler  kernel is defined by
\begin{equation*}
  ~~^{ABC}_{a}D^\alpha f(t)=\frac{B(\alpha)}{1-\alpha} \int_a^t f^\prime(x)E_\alpha\left(\frac{-\alpha}{1-\alpha}(t-x)^\alpha\right)dx,
\end{equation*}

\noindent and  the left Riemann-Liouville one by

\begin{equation*}\label{d2}
  ~~^{ABR}_{a}D^\alpha f(t)=\frac{B(\alpha)}{1-\alpha}\frac{d}{dt} \int_a^t f(x)E_\alpha\left(\frac{-\alpha}{1-\alpha}(t-x)^\alpha\right)dx,
 \end{equation*}
where $B(\alpha)>0$ is a normalization function with $B(0)=B(1)=1.$ In addition, the associated fractional integral is defined by
\begin{equation*}\label{d3}
  ~~^{AB}_{a}I^\alpha f(t)=\frac{1-\alpha}{B(\alpha)}f(t)+\frac{\alpha}{B(\alpha)}~_{a}I^\alpha f(t).
\end{equation*}
\end{definition}


If $f$ is defined on the interval $[a,b]$, then the action of the $Q-$operator is defined as $(Qf)(t)=f(a+b-t)$. From classical fractional calculus, it is known that $(_{a}I^\alpha Qf)(t)=Q(I_b^\alpha f)(t)$ and $(_{a}D^\alpha Qf)(t)=Q (D_b^\alpha f)(t)$. In \cite{Thabet1}, by making use of the $Q-$operator, the authors defined the right versions of the $ABR$ and $ABC$ fractional  derivatives and their corresponding integral as follows:

\begin{definition} (\cite{ Thabet1})
Let $f \in H^1(a,b),~a<b,~\alpha \in [0,1]$. Then the right Caputo fractional derivative  with Mittag-Leffler kernel is defined by
 $$~^{ABC}D^\alpha_b f(t)=-\frac{B(\alpha)}{1-\alpha} \int_t^b f^\prime(x)E_\alpha\left(\frac{-\alpha}{1-\alpha} (x-t)^\alpha\right)dx,$$
 and the right Riemann-Liouville one by
$$~^{ABR}D^\alpha_b f(t)=-\frac{B(\alpha)}{1-\alpha}\frac{d}{dt} \int_t^b f(x)E_\alpha\left(\frac{-\alpha}{1-\alpha}(x-t)^\alpha\right)dx.$$ In addition, the corresponding fractional integral is defined by
\begin{equation*}\label{d3}
 ~^{AB}I_b^\alpha f(t)=\frac{1-\alpha}{B(\alpha)}f(t)+\frac{\alpha}{B(\alpha)}I_{b}^\alpha f(t).
\end{equation*}
\end{definition}



The following function spaces were introduced in \cite{Thabet1} in order to present an integration by parts formula for  $ABR$ fractional derivatives. For $p\geq1$ and $\alpha >0$,

\begin{equation*}\label{n1}
  ~~^{AB}_{a}I^\alpha (L_p)=\{f: f=~^{AB}_{a}I^\alpha \varphi, ~~\varphi \in L_p(a,b)\},
\end{equation*}
and
\begin{equation*}\label{n2}
  ~^{AB}I_b^\alpha (L_p)=\{f: f=~^{AB}I_b^\alpha \phi, ~~\phi \in L_p(a,b)\}.
\end{equation*}

\noindent In \cite{Abdon, Thabet1}, it was shown that the left and right fractional   operators $~~^{ABR}_{a}D^\alpha$ and $~^{ABR}D_b^\alpha$ and their associated fractional integrals $~~^{AB}_{a}I^\alpha$ and $~^{AB}I_b^\alpha$ satisfy

 $$~~^{ABR}_{a}D^\alpha ~~^{AB}_{a}I^\alpha f(t)=f(t),\quad~~~^{ABR}D_b^\alpha ~^{AB}I_b^\alpha f(t)=f(t),$$
 and also
  \begin{equation}\label{satisfy2}
   ^{AB}_{a}I^\alpha ~^{ABR}_{a}D^\alpha  f(t)=f(t),\quad~~~^{AB}I_b^\alpha ~^{ABR}D_b^\alpha  f(t)=f(t).
  \end{equation}

\noindent Hence, from (\ref{satisfy2}), it follows that the function spaces $~~^{AB}_{a}I^\alpha (L_p)$ and $~^{AB}I_b^\alpha (L_p)$ are nonempty.

\begin{theorem} (\cite{Thabet1})(Integration by parts formula for $ABR$ fractional derivatives)\label{Integration by parts}

 Let $\alpha >0$, $p\geq 1,~q \geq 1$, and $\frac{1}{p}+\frac{1}{q}\leq 1+\alpha$ ($p\neq1$ and $q\neq1$ in case $\frac{1}{p}+\frac{1}{q}=1+\alpha$).
\begin{itemize}
  \item If  $\varphi(x) \in L_p(a,b) $ and $\psi(x) \in L_q(a,b)$, then
   \begin{eqnarray*} \label{IBP 1}
       \int_a^b \varphi(x) ~^{AB}_{a}I^\alpha\psi(x)dx= \int_a^b \psi(x) ~^{AB}I_b^\alpha\varphi(x)dx.
           \end{eqnarray*}


  \item If $f(x) \in ~^{AB}I_b^\alpha (L_p) $ and $g(x)\in ~^{AB}_{a}I^\alpha (L_q)$, then  $$\int_a^b f(x) ~^{ABR}_{a}D^\alpha g(x)dx=\int_a^b g(x) ~^{ABR}D_b^\alpha f(x) dx.$$
\end{itemize}
\end{theorem}

From \cite{Abdon}, we recall the  following relation between the left $ABR $ and $ABC$ fractional derivatives as
\begin{equation}\label{relation C and R}
  ~~^{ABC}_{0}D^\alpha f (t)=~^{ABR}_{0}D^\alpha f(t)-\frac{B(\alpha)}{1-\alpha} f(0)E_\alpha \left(-\frac{\alpha}{1-\alpha}t^\alpha\right).
\end{equation}

\noindent Right version of (\ref{relation C and R}) was proved in \cite{Thabet1} by making use of the $Q-$operator as follows:
\begin{equation}\label{right relation C and R}
  ~^{ABC}D_b^\alpha f(t)=~^{ABR}D_b^\alpha f(t)-\frac{B(\alpha)}{1-\alpha} f(b)E_\alpha \left(-\frac{\alpha}{1-\alpha}(b-t)^\alpha\right).
\end{equation}

From \cite{Antonyetal}, recall the left generalized fractional integral operator as
\begin{equation}\label{GIO}
 \textbf{ E}^\rho_{\alpha, \beta, \omega,a^+}\varphi(x)=\int_a^x (x-t)^{\beta-1} E_{\alpha,\beta}^\rho (\omega (x-t)^\alpha)\varphi(t) dt,~~x>a.
\end{equation}

\noindent Analogously, the right generalized fractional integral operator can be defined by

\begin{equation}\label{GIOr}
 \textbf{ E}^\rho_{\alpha, \beta, \omega,b^-}\varphi(x)=\int_x^b (t-x)^{\beta-1} E_{\alpha,\beta}^\rho (\omega (t-x)^\alpha)\varphi(t) dt,~~x<b
\end{equation}
(see also \cite{Thabet1}).


\begin{remark} \label{rem1} By means of (\ref{GIO}) and (\ref{GIOr}),  the $ABR$ and $ABC$ fractional derivatives can be expressed as

\begin{equation*}\label{q}
  ~~^{ABR}_{a}D^\alpha f(t)=\frac{B(\alpha)}{1-\alpha} \frac{d}{dt} \textbf{ E}^1_{\alpha,1, \frac{-\alpha}{1-\alpha},a^+}f(t),
\end{equation*}

\begin{equation*}\label{r}
  ~^{ABR}D_b^\alpha f(t)=\frac{-B(\alpha)}{1-\alpha}\frac{d}{dt} \textbf{ E}^1_{\alpha,1, \frac{-\alpha}{1-\alpha},b^-}f(t),
\end{equation*}
\begin{equation*}\label{t}
  ~~^{ABC}_{a}D^\alpha f(t)=\frac{B(\alpha)}{1-\alpha} \textbf{ E}^1_{\alpha,1, \frac{-\alpha}{1-\alpha},a^+}f^\prime(t),
\end{equation*}

\begin{equation*}\label{u}
  ~^{ABC}D_b^\alpha f(t)=\frac{-B(\alpha)}{1-\alpha} \textbf{ E}^1_{\alpha,1, \frac{-\alpha}{1-\alpha},b^-}f^\prime(t).
\end{equation*}
\end{remark}
\begin{proposition} (\cite{Thabet1}) (Integration by parts formula for $ABC$ fractional derivatives) \label{C by parts}

Let $f,g \in H^1(a,b)$ and $0<\alpha<1.$ Then we have
\begin{itemize}
  \item $\int_a^b g(t)~^{ABC}_{a}D^\alpha f(t)dt= \int_a^b f(t) ~^{ABR}D_b^\alpha g(t) dt+ \frac{B(\alpha)}{1-\alpha} f(t) \textbf{ E}^1_{\alpha,1, \frac{-\alpha}{1-\alpha},b^-}g (t)|_a^b$.
  \item $\int_a^b g(t)~^{ABC}D_b^\alpha f(t)dt= \int_a^b f(t) ~^{ABR}_{a}D^\alpha g(t) dt- \frac{B(\alpha)}{1-\alpha} f(t) \textbf{ E}^1_{\alpha,1, \frac{-\alpha}{1-\alpha},a^+}g (t)|_a^b$.
 \end{itemize}
\end{proposition}

The proof of Proposition \ref{C by parts} was presented in \cite{Thabet1} by making use of the relations (\ref{relation C and R}) and (\ref{right relation C and R}) and the $ABR$ integration by parts formula in Theorem \ref{Integration by parts}. Below, we present an alternative proof for Proposition \ref{C by parts}  by using an integration by parts formula for the generalized fractional integral operators defined in (\ref{GIO}) and (\ref{GIOr}) and the ordinary integration by parts.
\begin{lemma}\label{TR} Let $\alpha >0$, $p\geq 1,~q \geq 1$, and $\frac{1}{p}+\frac{1}{q}\leq 1+\alpha$ ($p\neq1$  and $q\neq1$ in  case $\frac{1}{p}+\frac{1}{q}=1+\alpha$). If $\varphi(x) \in L_p(a,b) $ and $\psi(x) \in L_q(a,b)$, then
   \begin{eqnarray} \label{gIBP 1}
            \nonumber
             \int_a^b \varphi(t) \textbf{ E}^1_{\alpha,1, \frac{-\alpha}{1-\alpha},a^+}\psi(t)dt = \int_a^b \psi(t) \textbf{ E}^1_{\alpha,1, \frac{-\alpha}{1-\alpha},b^-}\varphi(t)dt.
           \end{eqnarray}

\end{lemma}
\begin{proof}
The proof follows from the definition of the generalized fractional integral operators and interchanging the order of integration.
\end{proof}
Now, we present the alternative proof of  Proposition \ref{C by parts}:

\begin{proof}
Using Remark \ref{rem1}, Lemma \ref{TR}, and the ordinary integration by parts, we get

\begin{eqnarray*}
 \nonumber
  \int_a^b g(t)^{ABC}_{a}D^\alpha f(t) dt&=& \frac{B(\alpha)}{1-\alpha}\int_a^b g(t) \textbf{ E}^1_{\alpha,1, \frac{-\alpha}{1-\alpha},a^+}f^\prime(t)dt  \\ \nonumber
   &=& \frac{B(\alpha)}{1-\alpha}\int_a^b f^\prime(t)\textbf{ E}^1_{\alpha,1, \frac{-\alpha}{1-\alpha},b^-}g(t)dt \\ \nonumber
  &=&  \frac{B(\alpha)}{1-\alpha}f(t)\textbf{ E}^1_{\alpha,1, \frac{-\alpha}{1-\alpha},b^-}g(t)|_a^b\\ \nonumber
   &-& \frac{B(\alpha)}{1-\alpha}\int_a^b f(t)\frac{d}{dt}\textbf{ E}^1_{\alpha,1, \frac{-\alpha}{1-\alpha},b^-}g(t)dt\\ \nonumber
   &=& \frac{B(\alpha)}{1-\alpha} f(t)\textbf{ E}^1_{\alpha,1, \frac{-\alpha}{1-\alpha},b^-}g(t)|_a^b \\
    &+&\int_a^b f(t)~^{ABR}D_b^\alpha g(t)dt.
\end{eqnarray*}
Similarly, the proof of the second part is as follows:
\begin{eqnarray*}
 \nonumber
  \int_a^b g(t)^{ABC}D_b^\alpha f(t) dt&=& \frac{-B(\alpha)}{1-\alpha}\int_a^b g(t) \textbf{ E}^1_{\alpha,1, \frac{-\alpha}{1-\alpha},b^-}f^\prime(t)dt  \\ \nonumber
   &=& \frac{-B(\alpha)}{1-\alpha}\int_a^b f^\prime(t)\textbf{ E}^1_{\alpha,1, \frac{-\alpha}{1-\alpha},a^+}g(t)dt \\ \nonumber
  &=&  \frac{-B(\alpha)}{1-\alpha}f(t)\textbf{ E}^1_{\alpha,1, \frac{-\alpha}{1-\alpha},a^+}g(t)|_a^b\\ \nonumber
   &+& \frac{B(\alpha)}{1-\alpha}\int_a^b f(t)\frac{d}{dt}\textbf{ E}^1_{\alpha,1, \frac{-\alpha}{1-\alpha},a^+}g(t)dt\\ \nonumber
   &=& \frac{-B(\alpha)}{1-\alpha} f(t)\textbf{ E}^1_{\alpha,1, \frac{-\alpha}{1-\alpha},a^+}g(t)|_a^b \\
    &+&\int_a^b f(t)~^{ABR}_{a}D^\alpha g(t)dt.
\end{eqnarray*}

\end{proof}





\indent
Now, we recall some notations and basic definitions concerning the classical nabla discrete fractional  calculus. For more details, we refer the reader to \cite{Nabla, ThFer, dualCaputo, dualR, CP} and the references cited therein.

The functions we consider will be defined on  sets of the form
$$\mathbb{N}_a=\{a,a+1,a+2,...\},\qquad _{b}\mathbb{N}=\{...,b-2,b-1,b\},$$
where $a,b\in\mathbb{R},$ or a set of the form

$$\mathbb{N}_{a,b}=\{a,a+1,a+2,...,b\},$$
where $b-a$  is a positive integer.

\begin{definition} [\cite{Nabla, CP}] \label{rising}  
\noindent (i) For a natural number $m$ and $t\in\mathbb{R}$, the $m$ rising (ascending) factorial of $t$ is defined by

\begin{equation*}\label{rising 1}
    t^{\overline{m}}= \prod_{k=0}^{m-1}(t+k),~~~t^{\overline{0}}=1.
\end{equation*}

\noindent (ii) For any real number $\alpha$, the (generalized) rising function is defined by
\begin{equation*}\label{alpharising}
 t^{\overline{\alpha}}=\frac{\Gamma(t+\alpha)}{\Gamma(t)},~~~t \in \mathbb{R}\setminus \{...,-2,-1,0\},~~0^{\overline{\alpha}}=0.
\end{equation*}

\end{definition}



\indent

\begin{definition}[\cite{dualCaputo, dualR}]
\label{fractional sums}
For a function $f:\mathbb{N}_a\rightarrow \mathbb{R}$,
the  nabla left fractional sum of order $\alpha>0$ (starting from $a$) is given by
\begin{equation*}
_{a}\nabla^{-\alpha} f(t)=\frac{1}{\Gamma(\alpha)}
\sum_{s=a+1}^t(t-\rho(s))^{\overline{\alpha-1}}f(s),
\quad t \in \mathbb{N}_{a+1}.
\end{equation*}
 The   nabla  right fractional sum of order
$\alpha>0$ (ending at $b$) for $f:~_{b}\mathbb{N}\rightarrow \mathbb{R}$ is defined  by
\begin{equation*}
{\nabla_{b}^{-\alpha}} f(t)
=\frac{1}{\Gamma(\alpha)}
\sum_{s=t}^{b-1}(s-\rho(t))^{\overline{\alpha-1}}f(s)\\
=\frac{1}{\Gamma(\alpha)}
\sum_{s=t}^{b-1}(\sigma(s)-t)^{\overline{\alpha-1}}f(s),
\quad t \in {_{b-1}\mathbb{N}}.
\end{equation*}
\end{definition}

\begin{definition}[\cite{dualCaputo, dualR}]
For a function $f:\mathbb{N}_a\rightarrow \mathbb{R}$, the nabla left  Riemann-Liouville  fractional difference of order $0<\alpha<1$ (starting from $a$) is defined by
\begin{equation*}
_{a}\nabla^{\alpha} f(t)=\nabla  {_{a}}\nabla^{-(1-\alpha)}f(t)
= \nabla\left[\frac{1}{\Gamma(1-\alpha)}
\sum_{s=a+1}^t(t-\rho(s))^{\overline{-\alpha}}f(s)\right],
\quad t \in \mathbb{N}_{a+1},
\end{equation*}
and for $f:~_{b}\mathbb{N}\rightarrow \mathbb{R},$ the nabla right Riemann-Liouville   fractional difference of order $0<\alpha<1$  (ending at $b$)
 by
\begin{equation*}
{\nabla_{b}^{\alpha}} f(t)
= {_{\circleddash}\Delta} {\nabla_{b}^{-(1-\alpha)}}f(t)
=-\Delta\left[\frac{1}{\Gamma(1-\alpha)}
\sum_{s=t}^{b-1}(s-\rho(t))^{\overline{-\alpha}}f(s)\right],
\quad t \in {_{b-1}\mathbb{N}}.
\end{equation*}

\noindent In the above, $\rho$ and $\sigma$ are the backward and forward jump operators, respectively.

\begin{definition}[\cite{dualCaputo, dualR}]
For a function $f:\mathbb{N}_a\rightarrow \mathbb{R}$, the nabla left Caputo  fractional difference of order $0<\alpha<1$ (starting from $a$) is defined by $$(~^{C}~{_{a}}\nabla^\alpha f)(t)=~_{a}\nabla^{-(1-\alpha)} \nabla f(t),~~t \in \mathbb{N}_{a+1},$$
and for $f:~_{b}\mathbb{N}\rightarrow \mathbb{R},$ the nabla right Caputo  fractional difference of order $0<\alpha<1$ (ending at $b$) by

$$(~^{C}\nabla_{b}^\alpha f)(t)= \nabla_{b}^{-(1-\alpha)}~_{\ominus}\Delta f(t),~~t \in ~_{b-1}\mathbb{N}.$$
\end{definition}
\end{definition}

\indent





\begin{definition} \label{nDML} (Nabla Discrete Mittag-Leffler functions)(\cite{dualCaputo, dualR, Thsemi}) For $\lambda \in \mathbb{R}$, $|\lambda|<1,$ and $\alpha, \beta \in \mathbb{C}$ with $Re(\alpha)>0$, the nabla discrete  Mittag-Leffler function is defined by
\begin{equation*}
E_{\overline{\alpha, \beta}}(\lambda,z)= \sum_{k=0}^\infty \lambda^k
\frac{z^{\overline{k\alpha+\beta-1}}} {\Gamma(\alpha
k+\beta)}.
\end{equation*}
For $\beta=1$, it is written that
\begin{equation*} \label{nM22}
E_{\overline{\alpha}} (\lambda, z)\triangleq E_{\overline{\alpha, 1}}(\lambda, z)=  \sum_{k=0}^\infty \lambda^k
\frac{z^{\overline{k\alpha}}} {\Gamma(\alpha
k+1)}.
\end{equation*}

\end{definition}


\begin{definition}\label{DMLf} (\cite{Thabet2}) The nabla discrete generalized Mittag-Leffler function of three parameters $\alpha,~\beta,$ and $\rho$ is defined by

\begin{equation*}\label{dML3}
  E^\rho_{\overline{\alpha,\beta}} (\lambda, z)=\sum_{k=0}^\infty \lambda^k (\rho)_k
\frac{z^{\overline{k\alpha+\beta-1}}} { k! \Gamma(\alpha k+\beta)}.
\end{equation*}

\end{definition}

\noindent Notice that $E^1_{\overline{\alpha,\beta}} (\lambda, z)=E_{\overline{\alpha,\beta}} (\lambda, z)$.




Now, we review some main concepts concerning the nabla discrete fractional differences with discrete Mittag-Leffler kernels following the notations in \cite{Thabet2}.

\begin{definition} (\cite{Thabet2})
Assume $f:{\mathbb N}_{a}\rightarrow{\mathbb R}$ and $\alpha \in (0,1/2)$. Then the nabla discrete left Caputo fractional difference  in the sense of Atangana and Baleanu is defined by
\begin{equation*}
  ~~^{ABC}_{a}\nabla^\alpha f(t)=\frac{B(\alpha)}{1-\alpha} \sum_{s=a+1}^t\nabla f(s)E_{\overline{\alpha}}\left(\frac{ -\alpha}{1-\alpha}, t-\rho(s)\right),\quad t \in \mathbb{N}_{a+1},
\end{equation*}
and in the left Riemann-Liouville sense by
\begin{equation*}\label{d2}
  ~~^{ABR}_{a}\nabla^\alpha f(t)=\frac{B(\alpha)}{1-\alpha}\nabla_{t} \sum_{s=a+1}^t f(s)E_{\overline{\alpha}}\left(\frac{ -\alpha}{1-\alpha}, t-\rho(s)\right),\quad t \in \mathbb{N}_{a+1},
\end{equation*}
where  $B(\alpha)>0$ is a normalization function with $B(0)=B(1)=1.$ In addition, the associated fractional sum is defined by
\begin{equation*}\label{T4}
  ~^{AB}_{a}\nabla^{-\alpha} f(t)=\frac{1-\alpha}{B(\alpha)}f(t)+\frac{\alpha}{B(\alpha)} {~_{a}}\nabla^{-\alpha}f(t),\quad t \in \mathbb{N}_{a+1}.
\end{equation*}
\end{definition}

Similar to the continuous case, for a function $f$  defined on  $\mathbb{N}_{a,b}$, the action of the $Q-$operator is defined as $(Qf)(t)=f(a+b-t)$. From classical discrete fractional calculus, it is known that $(_{a}\nabla^{-\alpha} Qf)(t)=Q(\nabla_b^{-\alpha} f)(t)$ and $(_{a}\nabla^\alpha Qf)(t)=Q (\nabla_b^\alpha f)(t)$. In \cite{Thabet2}, by making use of the $Q-$operator, the authors defined the right versions of the $ABR$ and $ABC$ nabla fractional differences and their corresponding sum as follows:

\begin{definition} (\cite{Thabet2})
Assume $f:~_{b}{\mathbb N}\rightarrow{\mathbb R}$ and $\alpha \in (0,1/2)$. Then the nabla discrete right Riemann-Liouville fractional difference  with discrete Mittag-Leffler kernel is defined by
\begin{equation*}\label{nrd}
   ~~^{ABR}\nabla_b^\alpha f(t)=-\frac{B(\alpha)}{1-\alpha}\Delta_{t} \sum_{s=t}^{b-1} f(s)E_{\overline{\alpha}}\left(\frac{ -\alpha}{1-\alpha}, s-\rho(t)\right),\quad t \in ~_{b-1}\mathbb{N},
\end{equation*}
and the right Caputo one by
\begin{equation*}\label{Crd}
   ~~^{ABC}\nabla_b^\alpha f(t)=-\frac{B(\alpha)}{1-\alpha} \sum_{s=t}^{b-1} \Delta f(s)E_{\overline{\alpha}}\left(\frac{ -\alpha}{1-\alpha}, s-\rho(t) \right),\quad t \in ~_{b-1}\mathbb{N}.
\end{equation*}
In addition, the associated fractional sum is defined by
\begin{equation*}\label{nrd}
  ~~^{AB}\nabla_b^{-\alpha} f(t)=\frac{1-\alpha}{B(\alpha)}f(t)+\frac{\alpha}{B(\alpha)} \nabla^{-\alpha}_{b}f(t),\quad t \in ~_{b-1}\mathbb{N}.
\end{equation*}
\end{definition}





In \cite{Thabet2}, it was shown that  the left and right fractional difference operators $~^{ABR}_{a}\nabla^{\alpha}$ and $~^{ABR}\nabla_b^\alpha$ and their associated fractional sums $~^{AB}_{a}\nabla^{-\alpha}$ and $~^{AB}\nabla_b^{-\alpha}$ satisfy

$$ ~^{ABR}_{a}\nabla^{\alpha}~^{AB}_{a}\nabla^{-\alpha} f(t)=f(t),\quad  ~^{ABR}\nabla_b^\alpha ~^{AB}\nabla_b^{-\alpha }f(t)=f(t),$$
and also
$$~^{AB}_{a}\nabla^{-\alpha}~^{ABR}_{a}\nabla^{\alpha}f(t)=f(t),\quad ~^{AB}\nabla_b^{-\alpha} ~^{ABR}\nabla_b^{\alpha }f(t)=f(t).$$

From \cite{Thabet2}, we also recall the following relation between the left  $ABR$ and $ABC$ nabla fractional differences as
\begin{equation}\label{Tt}
  ~^{ABC}_{a}\nabla^{\alpha} f(t)=~^{ABR}_{a}\nabla^{\alpha} f(t)-f(a)\frac{B(\alpha)}{1-\alpha}E_{\overline{\alpha}}\left(\frac{ -\alpha}{1-\alpha},t-a\right).
\end{equation}




\noindent Right version of (\ref{Tt}) was proved in \cite{Thabet2} by making use of the Q-operator as follows:
\begin{equation}\label{Ttt}
  ~^{ABC}\nabla_b^{\alpha} f(t)=~^{ABR}\nabla_b^{\alpha} f(t)-f(b)\frac{B(\alpha)}{1-\alpha}E_{\overline{\alpha}}\left(\frac{ -\alpha}{1-\alpha},b-t\right).
\end{equation}


\begin{theorem} (\cite{Thabet2}) \label{by parts sums} (Integration by parts formula for $ABR$ fractional sums)

Assume  $f, g:{\mathbb N_{a,b}}\rightarrow{\mathbb R}$ and $\alpha \in (0,1/2)$. Then we have
\begin{eqnarray}
 \nonumber
 \sum_{s=a+1}^{b-1} g(s) ~^{AB}_{a}\nabla ^{-\alpha}f(s)=\sum_{s=a+1}^{b-1} f(s) ~^{AB}\nabla_b ^{-\alpha}g(s). 
 \end{eqnarray}
\end{theorem}

\begin{theorem} (\cite{Thabet2})(Integration by parts formula for $ABR$ fractional differences)\label{Rp}

Assume  $f, g:{\mathbb N_{a,b}}\rightarrow{\mathbb R}$ and $\alpha \in (0,1/2)$. Then we have

\begin{eqnarray*}
 \nonumber
 \sum_{s=a+1}^{b-1} f(s) ~^{ABR}_{a}\nabla ^{\alpha} g(s) &=& \sum_{s=a+1}^{b-1} g(s) ~^{ABR}\nabla_b ^{\alpha}f(s).
   \end{eqnarray*}
\end{theorem}

Before presenting  an integration by parts formula for the left $ABC$ fractional differences, we first recall the discrete versions of the left and right generalized fractional integral operators given in (\ref{GIO}) and (\ref{GIOr}).

\begin{definition} (\cite{Thabet2}) \label{def15}
\begin{itemize}
  \item The discrete (left) generalized fractional integral operator is defined by $$ \textbf{ E}^1_{\overline{\alpha, \beta}, \omega,a^+}\varphi (t)=\sum_{s=a+1}^t (t-\rho(s))^{\overline{\beta-1}} E_{\overline{\alpha,\beta}} (\omega, t-\rho(s)  )\varphi(s) ,~~t\in \mathbb{N}_a.$$
  \item  The discrete (right) generalized fractional integral operator is defined by
  $$\textbf{ E}^1_{\overline{\alpha, \beta}, \omega,b^-}\varphi (t)=\sum_{s=t}^{b-1} (s-\rho(t))^{\overline{\beta-1}} E_{\overline{\alpha,\beta}} (\omega, s-\rho(t))\varphi(s) ,~~t\in ~_{b}\mathbb{N}.$$
\end{itemize}
\end{definition}

\begin{remark} \label{rem2} By means of Definition \ref{def15},  the $ABR$ and $ABC$ fractional differences can be expressed as:

\begin{equation*}\label{dq}
  ~^{ABR}_{a}\nabla^\alpha f(t)=\frac{B(\alpha)}{1-\alpha} \nabla \textbf{ E}^1_{\overline{\alpha,1}, \frac{-\alpha}{1-\alpha},a^+}f(t),
\end{equation*}

\begin{equation*}\label{dr}
  ~^{ABR}\nabla_b^\alpha f(t)=\frac{-B(\alpha)}{1-\alpha}\Delta  \textbf{ E}^1_{\overline{\alpha,1}, \frac{-\alpha}{1-\alpha},b^-}f(t),
\end{equation*}
\begin{equation*}\label{dt}
  ~~^{ABC}_{a}\nabla^\alpha f(t)=\frac{B(\alpha)}{1-\alpha}  \textbf{ E}^1_{\overline{\alpha,1}, \frac{-\alpha}{1-\alpha},a^+}\nabla f (t),
\end{equation*}

\begin{equation*}\label{du}
  ~^{ABC}\nabla_b^{\alpha} f(t)=\frac{-B(\alpha)}{1-\alpha} \textbf{ E}^1_{\overline{\alpha,1}, \frac{-\alpha}{1-\alpha},b^-}\Delta f(t).
\end{equation*}
\end{remark}

\begin{lemma}\label{dTR} Assume  $f, g:{\mathbb N_{a,b}}\rightarrow{\mathbb R}$ and $\alpha\in(0,1/2)$. Then we have
  \begin{eqnarray} \label{gIBP 1}
          \nonumber
            \sum_{s=a+1}^{b-1} f(s)  \textbf{ E}^1_{\overline{\alpha, 1},\frac{-\alpha}{1-\alpha},a^+} g(s) = \sum_{s=a+1}^{b-1} g(s) \textbf{ E}^1_{\overline{\alpha,1},\frac{-\alpha}{1-\alpha},b^-} f(s).
           \end{eqnarray}



\end{lemma}
\begin{proof}
The proof follows from the definition of the generalized fractional sums and interchanging the order of summations.
\end{proof}

In \cite{Thabet2}, the authors presented integration by parts formulas for the $ABC$ nabla fractional differences by using Theorem \ref{Rp}, (\ref{Tt}) and (\ref{Ttt}). In what follows, we will present an integration by parts formula  for the left $ABC$ fractional differences by making use of Remark \ref{rem2}, Lemma \ref{dTR}, and the integration by parts in the ordinary difference calculus.

\begin{theorem}(Integration by parts formula for left $ABC$ fractional differences)\label{th9}

Assume  $f, g:{\mathbb N_{a,b}}\rightarrow{\mathbb R}$ and $\alpha \in (0,1/2)$. Then we have

\begin{equation*}\label{Cp1}
  \sum_{s=a+1}^{b-1}f(s) ~^{ABC}_{a}\nabla^\alpha g(s)=\sum_{s=a+1}^{b-1}g(s-1) ~^{ABR}\nabla_{b}^\alpha f(s-1)+ g(t)\frac{B(\alpha)}{1-\alpha} \textbf{ E}^1_{\overline{\alpha, 1}, \frac{-\alpha}{1-\alpha},b^-}f (t)|_a^{b-1}.
\end{equation*}
\end{theorem}

\noindent In the above, it is easy to see that $\textbf{ E}^1_{\overline{\alpha, 1},\frac{-\alpha}{1-\alpha},b^-}f (b-1)=(1-\alpha)f(b-1)$.


\begin{remark} \label{rem3}
The integration by parts formula in Theorem \ref{th9}  can now be stated as follows:
\begin{equation*}\label{Cp11}
  \sum_{s=a+1}^{b}f(s) ~^{ABC}_{a}\nabla^\alpha g(s)=\sum_{s=a+1}^{b}g(s-1) ~^{ABR}\nabla_{b+1}^\alpha f(s-1)+ g(t)\frac{B(\alpha)}{1-\alpha} \textbf{ E}^1_{\overline{\alpha, 1}, \frac{-\alpha}{1-\alpha},b+1^-}f (t)|_a^{b}.
\end{equation*}
\end{remark}

\section{Main Results}


Denoting the $SL$ operator as
\begin{equation*}\label{sadfe1}
 L _{1} x(t)=~~^{ABR}_{a}D^\alpha (p(t)~^{ABR}D^\alpha_b x(t))+q(t)x(t),
\end{equation*}
consider the fractional $SLE$

\begin{equation}
   ~~^{ABR}_{a}D^\alpha(p(t)~^{ABR}D^\alpha_b x(t))+q(t)x(t)=\lambda r(t) x(t),\quad t\in (a,b),\label{cslp}
\end{equation}
where  $\alpha\in(0,1),$ $p(t)>0,$ $ r(t)>0\,\forall t\in [a,b]$, $p, q, r$ are real valued continuous functions on the interval $[a,b].$



\begin{theorem}\label{sadfe1}
The  fractional $SL$ operator $L_1$ is self-adjoint with respect to the inner product
$$<u,v>=\int_{a}^{b} \overline{u(t)}v(t)\,dt.$$
\end{theorem}

\begin{proof}
We have
\begin{eqnarray}
v(t)\overline{L_1 u(t)}&=&v(t)~~^{ABR}_{a}D^\alpha(p(t)~^{ABR}D^\alpha_b \overline{u}(t))+q(t)\overline{u}(t)v(t)\label{eqn21}\\
\overline{u}(t)L_1 v(t)&=&\overline{u}(t)~~^{ABR}_{a}D^\alpha(p(t)~^{ABR}D^\alpha_b v(t))+q(t)\overline{u}(t)v(t)\label{eqn11}.
\end{eqnarray}


 \noindent Subtracting (\ref{eqn11}) from (\ref{eqn21}), we have

\begin{equation*}
v(t)\overline{L_1 u(t)}-\overline{u}(t)L_1 v(t)=v(t)~~^{ABR}_{a}D^\alpha(p(t)~^{ABR}D^\alpha_b \overline{u}(t))-\overline{u}(t)~~^{ABR}_{a}D^\alpha(p(t)~^{ABR}D^\alpha_b v(t)). \label{eqn31}
\end{equation*}


\noindent Integrating  from $a$ to $b$,  we get

\begin{eqnarray*}
& &\int_{a}^{b} \left(v(t)\overline{L_1 u(t)}-\overline{u}(t)L_1 v(t)\right)\,dt=\nonumber\\  & &\quad \int_{a}^{b}\left(v(t)~~^{ABR}_{a}D^\alpha(p(t)~^{ABR}D^\alpha_b \overline{u}(t))-\overline{u}(t)~~^{ABR}_{a}D^\alpha(p(t)~^{ABR}D^\alpha_b v(t))\right)\,dt. \label{eqn31}
\end{eqnarray*}


\noindent Now, by applying the integration by parts formula in Theorem \ref{Integration by parts}, we obtain

\begin{eqnarray*}
 \int_{a}^{b} \left(v(t)\overline{L_1 u(t)}-\overline{u}(t)L_1 v(t)\right)\,dt& = & \int_{a}^{b} p(t)~^{ABR}D^\alpha_b \overline{u}(t) ~~^{ABR}D_b^\alpha v(t)\,dt\\ & &-
\int_{a}^{b} p(t)~^{ABR}D^\alpha_b v(t) ~~^{ABR}D_b^\alpha \overline{u}(t)\,dt\\ &=& 0.
\end{eqnarray*}

\noindent Hence, $<L_1u,v>=<u,L_1v>.$ That is, $L_1$ is self-adjoint.
\end{proof}

\begin{theorem}\label{selfadjoint2}
The  eigenvalues of the  $SLE$ (\ref{cslp}) are real.
\end{theorem}

\begin{proof}

Assume that $\lambda$ is the eigenvalue for (\ref{cslp})  corresponding to eigenfunction $x.$ Then $x$ and its complex conjugate $\overline{x}$ satisfy
\begin{equation}\label{xxxx}
 L _{1} x(t)=\lambda r(t)x(t),
 \end{equation}
and
 \begin{equation}\label{xxx}
 L _{1} \overline{x}(t)=\overline{\lambda} r(t)\overline{x}(t),
 \end{equation}
 respectively.
 We multiply (\ref{xxxx}) by $\overline{x}(t)$ and  (\ref{xxx}) by $x(t),$ respectively, and subtract to obtain
 \begin{equation*}
 (\overline{\lambda}-\lambda)r(t)x(t)\overline{x}(t)=x(t)L _{1}\overline{x}(t)-\overline{x}(t) L _{1} x(t).
 \end{equation*}

\noindent Integrating from $a$ to $b,$ and by using the fact that $L_1$ is self-adjoint, we get

\begin{equation*}
(\overline{\lambda}-\lambda)\int_{a}^{b}r(t)|x(t)|^2\,dt=0,
\end{equation*}
and since $x$ is a nontrivial solution and $r(t)>0,$  we conclude that $\lambda=\overline{\lambda}.$
\end{proof}

\begin{theorem}\label{selfadjoint3}
The eigenfunctions, corresponding to distinct eigenvalues of the $SLE$ (\ref{cslp}) are orthogonal  with respect to the weight function $r$ on $[a,b]$ that is

$$<x_{\lambda_1},x_{\lambda_2}>=\int_{a}^{b} r(t)x_{\lambda_1}(t)x_{\lambda_2}(t)\,dt=0,\quad \lambda_1 \ne \lambda_2,$$
when the functions $x_{\lambda_i}$  correspond to eigenvalues $\lambda_i,\,i=1,2.$
\end{theorem}

\begin{proof}

Let $\lambda_1$  and $\lambda_2$ be two distinct eigenvalues of (\ref{cslp}) corresponding to the eigenfunctions $x_{\lambda_1}$ and $x_{\lambda_2}$, respectively. Then we have

\begin{eqnarray}
~~^{ABR}_{a}D^\alpha(p(t)~~^{ABR}D_b^\alpha x_{\lambda_1}(t))+q(t)x_{\lambda_1}(t)&=&\lambda_1 r(t)x_{\lambda_1}(t)\label{cev1}\\
~~^{ABR}_{a}D^\alpha(p(t)~~^{ABR}D_b^\alpha x_{\lambda_2}(t))+q(t)x_{\lambda_2}(t)&=&\lambda_2 r(t)x_{\lambda_2}(t)\label{cev2}.
\end{eqnarray}

\noindent  We multiply (\ref{cev1}) and (\ref{cev2}) by $x_{\lambda_2}(t)$ and $x_{\lambda_1}(t),$ respectively, and subtract  the results to obtain
 $$(\lambda_1-\lambda_2) r(t)x_{\lambda_1}(t)x_{\lambda_2}(t)=x_{\lambda_2}(t)~~^{ABR}_{a}D^\alpha(p(t)~~^{ABR}D_b^\alpha x_{\lambda_1}(t))-x_{\lambda_1}(t)~~^{ABR}_{a}D^\alpha(p(t)~~^{ABR}D_b^\alpha x_{\lambda_2}(t)). $$

\noindent Now, integrating from $a$ to $b,$ and using the integration by parts formula in Theorem \ref{Integration by parts} to the right side of the equation, we get

\begin{equation*}
(\lambda_{1}-\lambda_2)\int_{a}^{b}r(t)x_{\lambda_1}(t)x_{\lambda_2}(t)\,dt=0.
\end{equation*}
Since $\lambda_{1}\neq\lambda_{2},$ it follows that
\begin{equation*}
\int_{a}^{b}r(t)x_{\lambda_1}(t)x_{\lambda_2}(t)\,dt=0,
\end{equation*}
which completes the proof.
\end{proof}

Now, consider the nabla discrete fractional $SL$ operator
\begin{equation*}
 L_2 x(t)=~^{ABR}_{a}\nabla^{\alpha}(p(t)~~^{ABR}\nabla_b^\alpha x(t))+q(t)x(t),
\end{equation*}
and the corresponding $SLE$
\begin{equation}\label{dslp}
      ~^{ABR}_{a}\nabla^{\alpha}(p(t)~~^{ABR}\nabla_b^\alpha x(t))+q(t)x(t)=\lambda r(t) x(t),\quad t\in \mathbb{N}_{a+1,b-1},
\end{equation}
where  $\alpha\in (0,1/2),$ $p(t)>0,$ $ r(t)>0\,\forall t\in \mathbb{N}_{a,b}$, $p, q, r$ are real valued functions on $\mathbb{N}_{a,b}.$



\begin{theorem}
The  discrete fractional $SL$ operator $L_2$ is self-adjoint with respect to the inner product
$$<u,v>=\sum_{t=a+1}^{b-1} \overline{ u(t)}v(t).$$
\end{theorem}

\begin{proof}
The proof  is similar to that of Theorem \ref{sadfe1}. However, it follows by making use of the discrete fractional integration by parts in Theorem \ref{Rp}. The details are left to the reader.
\end{proof}




\begin{theorem}
The  eigenvalues of the  $SLE$ (\ref{dslp}) are real.
\end{theorem}

\begin{proof}
The proof  is similar to that of Theorem \ref{selfadjoint2}. The details are left to the reader.
\end{proof}


\begin{theorem}
The eigenfunctions, corresponding to distinct eigenvalues of the  $SLE$ (\ref{dslp}) are orthogonal  with respect to the weight function $r$ on $\mathbb{N}_{a,b}$ that is

$$<x_{\lambda_1},x_{\lambda_2}>=\sum_{t=a+1}^{b-1} r(t)x_{\lambda_1}(t)x_{\lambda_2}(t)=0,\quad \lambda_1 \ne \lambda_2,$$
when the functions $x_{\lambda_i}$  correspond to eigenvalues $\lambda_i,\,i=1,2.$
\end{theorem}

\begin{proof}
The proof  is similar to that of Theorem \ref{selfadjoint3}. However, it follows by making use of the discrete fractional integration by parts in Theorem \ref{Rp}. The details are left to the reader.
\end{proof}





Denoting the $SL$ operator as

$$~^{C}L_1 x(t)=~~^{ABC}_{a}D^\alpha (p(t)~^{ABR}D^\alpha_b x(t))+q(t)x(t),$$

\noindent consider the $ABC$ type fractional $SLE$ 

\begin{equation}\label{LC}
 ~~^{ABC}_{a}D^\alpha (p(t)~^{ABR}D^\alpha_b x(t))+q(t)x(t)=\lambda r(t) x(t),~~t\in (a,b),
\end{equation}
where $\alpha\in(0,1),$ $p(t)\ne 0,$ $ r(t)>0\,\forall t\in [a,b]$, $p, q, r$ are real valued continuous functions on the interval $[a,b]$, and the boundary conditions:

\begin{equation}\label{Bc1}
 c_1 \textbf{ E}^1_{\alpha,1, \frac{-\alpha}{1-\alpha},b^-}x(a)+ c_2 ~^{ABR}D_b^\alpha x(a)=0,
\end{equation}

\begin{equation}\label{Bc2}
 d_1 \textbf{ E}^1_{\alpha,1, \frac{-\alpha}{1-\alpha},b^-}x(b)+ d_2 ~^{ABR}D_b^\alpha x(b)=0,
\end{equation}
where  $c_1^2+c_2 ^2\ne0$ and  $d_1^2+d_2^2\ne0$.


\begin{theorem} \label{EV}
The eigenvalues of the $SLP$ (\ref{LC})-(\ref{Bc2})are real.
\end{theorem}
\begin{proof}
By the help of the integration by parts formula in Proposition \ref{C by parts}, we observe that
$$ \int_a^b u(t) ^{C}L_1 v(t) dt=\int_a^b q(t) u(t)v(t)dt+$$
\begin{equation}\label{Obs}
 \int_a^b p(t) ~^{ABR}D_b^\alpha v(t)~^{ABR}D_b^\alpha u(t)dt+\frac{B(\alpha)}{1-\alpha}p(t)\textbf{ E}^1_{\alpha,1, \frac{-\alpha}{1-\alpha},b^-}u (t)~^{ABR}D_b^\alpha v(t)|_a^b.
\end{equation}
Assume that $\lambda$ is the eigenvalue for (\ref{LC})-(\ref{Bc2}) corresponding to eigenfunction $x$. Then, $x$ and its complex conjugate $\overline{x}$  satisfy
\begin{equation}\label{s1}
  ^{C}L_1 x(t)= \lambda r(t) x(t),
\end{equation}
\begin{equation}\label{s11}
 c_1\textbf{ E}^1_{\alpha,1, \frac{-\alpha}{1-\alpha},b^-}x(a)+ c_2 ~^{ABR}D_b^\alpha x(a)=0,
\end{equation}

\begin{equation}\label{s12}
 d_1 \textbf{ E}^1_{\alpha,1, \frac{-\alpha}{1-\alpha},b^-}x(b)+ d_2 ~^{ABR}D_b^\alpha x(b)=0,
\end{equation}
and
\begin{equation}\label{cs1}
  ^{C}L_1 \overline{x}(t)= \overline{\lambda} r(t) \overline{x}(t),
\end{equation}
\begin{equation}\label{cs11}
 c_1 \textbf{ E}^1_{\alpha,1, \frac{-\alpha}{1-\alpha},b^-}\overline{x}(a)+ c_2 ~^{ABR}D_b^\alpha \overline{x}(a)=0,
\end{equation}

\begin{equation}\label{cs12}
 d_1 \textbf{ E}^1_{\alpha,1, \frac{-\alpha}{1-\alpha},b^-}\overline{x}(b)+ d_2 ~^{ABR}D_b^\alpha \overline{x}(b)=0,
\end{equation}
with $c_1^2+c_2 ^2\ne 0$ and  $d_1^2+d_2^2\ne 0$.
We multiply (\ref{s1}) by $\overline{x}(t)$ and (\ref{cs1}) by $x(t)$, respectively, and subtract to obtain
\begin{equation*}\label{oo}
  (\overline{\lambda}-\lambda)r(t) x(t) \overline{x}(t)=x(t)^{C}L_1 \overline{x}(t)- \overline{x}(t)^{C}L_1x(t).
\end{equation*}

Now, integrate over the interval $[a,b]$ and apply (\ref{Obs}) with $u(t)=x(t)$ and $v(t)=\overline{x}(t)$  and vice versa to obtain
$$(\overline{\lambda}-\lambda)\int_a^b r(t)|x(t)|^2 dt=$$
\begin{equation*}\label{xx}
  \frac{B(\alpha)}{1-\alpha} p(b)\left[\textbf{ E}^1_{\alpha,1, \frac{-\alpha}{1-\alpha},b^-}x(b) ~^{ABR}D_b^\alpha \overline{x}(b)- \textbf{ E}^1_{\alpha,1, \frac{-\alpha}{1-\alpha},b^-}\overline{x}(b) ~^{ABR}D_b^\alpha x(b)\right]+
\end{equation*}
\begin{equation*}\label{xx1}
 \frac{B(\alpha)}{1-\alpha} p(a)\left[ \textbf{ E}^1_{\alpha,1, \frac{-\alpha}{1-\alpha},b^-}\overline{x}(a) ~^{ABR}D_b^\alpha x(a)- \textbf{ E}^1_{\alpha,1, \frac{-\alpha}{1-\alpha},b^-}x(a) (~^{ABR}D_b^\alpha \overline{x}(a) \right].
\end{equation*}
Finally, by making use of the boundary conditions (\ref{s11}), (\ref{s12}), (\ref{cs11}), and (\ref{cs12}), we obtain
$$(\overline{\lambda}-\lambda)\int_a^br(t)|x(t)|^2 dt=0.$$
Because $x$ is a nontrivial solution and $r(t)>0$, we conclude that $\lambda=\overline{\lambda}.$
\end{proof}

\begin{theorem} \label{sonra}
The eigenfunctions, corresponding to distinct eigenvalues of  the $SLP$ (\ref{LC})-(\ref{Bc2}) are orthogonal with respect to  the weight function $r$ on $[a,b]$ that is
$$<x_{\lambda_1},x_{\lambda_2}>=\int_{a}^{b} r(t)x_{\lambda_1}(t)x_{\lambda_2}(t)\,dt=0,\quad \lambda_1 \ne \lambda_2,$$
when the functions $x_{\lambda_i}$  correspond to eigenvalues $\lambda_i,\,i=1,2.$
\end{theorem}

\begin{proof}
Assume that $x_{\lambda_1}$ and $x_{\lambda_2}$ are eigenfunctions  corresponding to two distinct eigenvalues $\lambda_1$ and $\lambda_2,$ respectively. Then we have
\begin{equation}\label{os1}
  ^{C}L_1 x_{\lambda_1}(t)= \lambda_1 r(t) x_{\lambda_1}(t),
\end{equation}
\begin{equation}\label{os11}
 c_1 \textbf{ E}^1_{\alpha,1, \frac{-\alpha}{1-\alpha},b^-}x_{\lambda_1}(a)+ c_2 ~^{ABR}D_b^\alpha x_{\lambda_1}(a)=0,
\end{equation}

\begin{equation}\label{os12}
 d_1 \textbf{ E}^1_{\alpha,1, \frac{-\alpha}{1-\alpha},b^-}x_{\lambda_1}(b)+ d_2 ~^{ABR}D_b^\alpha x_{\lambda_1}(b)=0,
\end{equation}
and
\begin{equation}\label{ocs1}
  ^{C}L_1 x_{\lambda_2}(t)= \lambda_2 r(t) x_{\lambda_2}(t),
\end{equation}
\begin{equation}\label{ocs11}
 c_1 \textbf{ E}^1_{\alpha,1, \frac{-\alpha}{1-\alpha},b^-}x_{\lambda_2}(a)+ c_2 ~^{ABR}D_b^\alpha x_{\lambda_2}(a)=0,
\end{equation}
\begin{equation}\label{ocs12}
 d_1 \textbf{ E}^1_{\alpha,1, \frac{-\alpha}{1-\alpha},b^-}x_{\lambda_2}(b)+ d_2 ~^{ABR}D_b^\alpha x_{\lambda_2}(b)=0.
\end{equation}
We multiply (\ref{os1}) by $x_{\lambda_2}(t)$ and (\ref{ocs1}) by $x_{\lambda_1}(t),$ respectively, and subtract to obtain
\begin{equation*}\label{wereach}
  x_{\lambda_2}(t)~^{C}L_1 x_{\lambda_1}(t)-x_{\lambda_1}(t)~^{C}L_1 x_{\lambda_2}(t)= (\lambda_1-\lambda_2) r(t) x_{\lambda_1}(t)x_{\lambda_2}(t).
\end{equation*}
Now, we integrate over the interval $[a,b]$ and apply (\ref{Obs}) with $u(t)=x_{\lambda_2}(t)$ and $v(t)=x_{\lambda_1}(t)$ and vice versa to obtain

$$(\lambda_1-\lambda_2)\int_a^b r(t) x_{\lambda_1}(t)x_{\lambda_2}(t)\,dt=$$
\begin{equation*}\label{xxxxxxx}
  \frac{B(\alpha)}{1-\alpha} p(b)\left[\textbf{ E}^1_{\alpha,1, \frac{-\alpha}{1-\alpha},b^-}x_{\lambda_2}(b) ~^{ABR}D_b^\alpha x_{\lambda_1}(b)- \textbf{ E}^1_{\alpha,1, \frac{-\alpha}{1-\alpha},b^-}x_{\lambda_1}(b) ~^{ABR}D_b^\alpha x_{\lambda_2}(b)\right]+
\end{equation*}
\begin{equation*}\label{xx1}
 \frac{B(\alpha)}{1-\alpha} p(a)\left[ \textbf{ E}^1_{\alpha,1, \frac{-\alpha}{1-\alpha},b^-}x_{\lambda_1}(a) ~^{ABR}D_b^\alpha x_{\lambda_2}(a)- \textbf{ E}^1_{\alpha,1, \frac{-\alpha}{1-\alpha},b^-}x_{\lambda_2}(a) ~^{ABR}D_b^\alpha x_{\lambda_1}(a) \right].
\end{equation*}
Finally, by using the boundary conditions (27), (28), (30), and (31), we conclude that
$$(\lambda_1-\lambda_2)\int_a^b r(t)x_{\lambda_1}(t) x_{\lambda_2}(t) dt=0,$$ and hence, as $\lambda_1\neq \lambda_2$,  $<x_{\lambda_1},x_{\lambda_2}>=0$.
\end{proof}
\indent
\begin{remark}\label{rem10}
The integration by parts formula in  the second part of Proposition \ref{C by parts} suggests the following $SLE$:
\begin{equation}\label{LR}
 ~~^{ABC}D^\alpha_b (p(t)~^{ABR}_{a}D^\alpha x(t))+q(t)x(t)=\lambda r(t) x(t),~~t\in (a,b),
\end{equation}
with the boundary conditions:

\begin{equation}\label{RBc1}
 c_1\textbf{ E}^1_{\alpha,1, \frac{-\alpha}{1-\alpha},a^+}x(a)+ c_2 ~^{ABR}_{a}D^\alpha x(a)=0,
\end{equation}

\begin{equation}\label{RBc2}
 d_1 \textbf{ E}^1_{\alpha,1, \frac{-\alpha}{1-\alpha},a^+}x(b)+ d_2 ~^{ABR}_{a}D^\alpha x(b)=0,
\end{equation}
where $c_1^2+c_2 ^2\ne0$ and  $d_1^2+d_2^2\ne0$.
Similar properties can  also be proved for such a $SLP$ (\ref{LR})-(\ref{RBc2}) as proved for the $SLP$ (\ref{LC})-(\ref{Bc2}).
\end{remark}

Denoting the SL operator as
\begin{equation*}
 ~^{dC}L_1x(t)=~~^{ABC}_{a}\nabla^\alpha (p(t)~^{ABR}\nabla^\alpha_b x(t))+q(t)x(t),
\end{equation*}
consider the following nabla discrete  $ABC$ type $SLE$:
\begin{equation}\label{dLC}
 ~~^{ABC}_{a}\nabla^\alpha (p(t)~^{ABR}\nabla^\alpha_b x(t))+q(t)x(t)=\lambda r(t) x(t),~~t\in \mathbb{N}_{a+1,b-1},
\end{equation}
where $\alpha\in(0,1/2),$ $p(t)\ne 0,$ $ r(t)>0\,\forall t\in \mathbb{N}_{a,b-1},$ $p, q, r$ are real valued  functions on $\mathbb{N}_{a,b-1},$ and the boundary conditions:
\begin{equation}\label{dBc1}
 c_1 \textbf{ E}^1_{\overline{\alpha,1}, \frac{-\alpha}{1-\alpha},b^-}x(a)+ c_2 ~^{ABR}\nabla_b^\alpha x(a)=0,
\end{equation}

\begin{equation}\label{dBc2}
 d_1 \textbf{ E}^1_{\overline{\alpha,1}, \frac{-\alpha}{1-\alpha},b^-}x(b-1)+ d_2 ~^{ABR}\nabla_b^\alpha x(b-1)=0,
\end{equation}
where $c_1^2+c_2 ^2\ne0$ and  $d_1^2+d_2^2\ne0$.

Now, we make use of the integration by parts formula in Theorem \ref{th9} to prove  the corresponding properties for the $SLP$ (\ref{dLC})-(\ref{dBc2}). For such a discrete fractional $SLP$, we use the following inner product
\begin{equation}\label{ip}
  <f,g>= \sum_{t=a+1}^{b-1}r(t) f(t)g(t).
\end{equation}

\begin{remark} \label{rem4}
If we want to use the integration by parts formula stated in Remark \ref{rem3}, then we consider the following nabla discrete fractional $SLE$
 \begin{equation*}\label{dLC3}
 ~^{dC}L _{2} x(t)=~~^{ABC}_{a}\nabla^\alpha (p(t)~^{ABR}\nabla^\alpha_{b+1} x(t))+q(t)x(t)=\lambda r(t) x(t),~~t\in \mathbb{N}_{a+1,b},
\end{equation*}
where $\alpha\in(0,1/2),$ $p(t)\ne 0,$ $ r(t)>0\,\forall t\in \mathbb{N}_{a,b},$ $p, q, r$ are real valued  functions on $\mathbb{N}_{a,b},$ and the boundary conditions:
\begin{equation*}\label{dBc13}
 c_1 \textbf{ E}^1_{\overline{\alpha,1}, \frac{-\alpha}{1-\alpha},b+1^-}x(a)+ c_2 ~^{ABR}\nabla_{b+1}^\alpha x(a)=0,
\end{equation*}

\begin{equation*}\label{dBc23}
 d_1 \textbf{ E}^1_{\overline{\alpha,1}, \frac{-\alpha}{1-\alpha},b+1^-}x(b)+ d_2 ~^{ABR}\nabla_{b+1}^\alpha x(b)=0,
\end{equation*}
where $c_1^2+c_2 ^2\ne0$ and  $d_1^2+d_2^2\ne0$, together with the inner product
\begin{equation*}\label{ip3}
   <f,g>= \sum_{t=a+1}^{b}r(t) f(t)g(t).
\end{equation*}
\end{remark}

\begin{theorem} \label{dEV}
The eigenvalues of the nabla discrete fractional $SLP$ (\ref{dLC})-(\ref{dBc2})are real.
\end{theorem}
\begin{proof}
The proof is similar to that of Theorem \ref{EV}. However, it follows by making use of the discrete fractional integration by parts in Theorem \ref{th9}. The details are left to the reader.
\end{proof}

\begin{theorem}
The eigenfunctions, corresponding to distinct eigenvalues of the  $SLP$ (\ref{dLC})-(\ref{dBc2}) are orthogonal with respect to  the weight function $r$ on $[a,b]$ that is
$$
  <x_{\lambda_1}, x_{\lambda_2}>=\sum_{t=a+1}^{b-1} r(t) x_{\lambda_1}(t)x_{\lambda_2}(t)  =0,\quad {\lambda_1}\ne{\lambda_2},$$
when the functions $x_{\lambda_i}$  correspond to eigenvalues $\lambda_i,\,i=1,2.$

\end{theorem} \label{dsonra}
\begin{proof}
The proof is similar to that of Theorem \ref{sonra}. However, it follows by making use of the discrete fractional integration by parts in Theorem \ref{th9} and through the use of the inner product defined in (\ref{ip}). The details are left to the reader.
\end{proof}

\begin{remark}
From the proofs of Theorem  \ref{EV} and Theorem \ref{dEV}, it follows that the  operators $~^{C}L_1$ and its discrete version $~^{dC}L_1$ are self-adjoint. Similarly, from Remark \ref{rem4},  we see that the discrete operator $~^{dC}L_2$ is also self-adjoint.
\end{remark}

\section{The higher order discrete fractional $SLE$ and an open problem}
In the previous section, the values of $\alpha$ in the discrete $SLE$ are  taken in the interval $(0,\frac{1}{2})$  in order to guarantee the convergence of the discrete Mittag-Leffler kernel used to define the $ABR$ and $ABC$ fractional differences. Therefore, the ordinary difference $SLE$ can not be obtained as  $\alpha\rightarrow 1^-$. For this purpose, we shall recommend for a discrete fractional $SLE$ with the values of $\alpha\in(1,\frac{3}{2})$ so that the ordinary difference case can be obtained as $\alpha\rightarrow 1^+$.
For the main concepts regarding to higher order fractional calculus with discrete Mittag-Leffler kernel, we refer to \cite{Th EPJ}. For  higher order fractional operators with non-singular Mittag-Leffler kernels, we refer the reader to \cite{JIA Lyapunov}, where a Lyapunov type inequality was formulated for a $BVP$ of order $2<\alpha<3$ and the ordinary Lyapunov inequality was obtained as $\alpha\rightarrow 2^+$. Using Definition 2.1. in \cite{Th EPJ} and Definition \ref{def15}, Lemma \ref{dTR},  and Remark \ref{rem2} in our paper that if $\alpha \in(1,2),$ $\beta =\alpha-1 \in (0,1)$,
$\lambda_\beta=\frac{-\beta}{1-\beta}=-\frac{\alpha-1}{2-\alpha}$ and $|\lambda_\beta|<1$ if $\alpha \in (1,\frac{3}{2})$, then the following four results hold:
\begin{equation*}\label{w1}
  ~^{ABC}_{a}\nabla^\alpha f(t)=~^{ABC}_{a}\nabla^\beta \nabla f(t)=\frac{B(\alpha-1)}{2-\alpha}\textbf{ E}^1_{\overline{\beta,1}, \lambda_\beta,a^+}\nabla^2 f(t),
\end{equation*}
\begin{equation*}\label{w2}
  ~^{ABR}_{a}\nabla^\alpha f(t)=~^{ABR}_{a}\nabla^\beta \nabla f(t)=\frac{B(\alpha-1)}{2-\alpha}\nabla \textbf{ E}^1_{\overline{\beta,1}, \lambda_\beta,a^+}\nabla f(t),
\end{equation*}

\begin{equation*}\label{w1}
  ~^{ABC}\nabla_b^\alpha f(t)=~^{ABC}\nabla_b^\beta (-\Delta f)(t)=\frac{B(\alpha-1)}{2-\alpha}\textbf{ E}^1_{\overline{\beta,1}, \lambda_\beta,b^-}\Delta^2 f(t),
\end{equation*}
\begin{equation*}\label{w2}
  ~^{ABR}\nabla_b^\alpha f(t)=~^{ABR}\nabla_b^\beta (-\Delta f)(t)=\frac{B(\alpha-1)}{2-\alpha}\Delta \textbf{ E}^1_{\overline{\beta,1}, \lambda_\beta,b^-}\Delta f(t).
\end{equation*}

\indent
\begin{remark}\label{finalrem}
Notice that the ordinary $SLE$ can be obtained in the non-discrete case as $\alpha\rightarrow 1^-$. In fact, since non-discrete Mittag-Leffler kernels do not have a convergence problem, the $ABR$ and $ABC$ fractional operators become well-defined  for any $\alpha \in (0,1]$. For that reason, we present the following open problem in the discrete case.
\end{remark}
\textbf{Open problem:}
Depending on the above discussion, can we present integration by parts formulas for the $ABR$  and $ABC$ fractional differences of order $\alpha \in (1,\frac{3}{2})$ which will be used to prove the self-adjointness, eigenvalue, and eigenfunction properties of suitable fractional difference $SLEs$ with proper boundary conditions?

\section{Acknowledgements}
This study was supported by The Scientific and Technological Research Council of Turkey while the first author visiting the University of Nebraska-Lincoln. The second author would like to thank Prince Sultan University for funding this work through the research group
Nonlinear Analysis Methods in Applied Mathematics (NAMAM) group number RG-DES-2017-01-17.

\end{document}